\documentclass[10pt,a4paper]{article}

\usepackage[utf8]{inputenc}
\usepackage[T1]{fontenc}
\usepackage{lmodern}

\usepackage{graphicx}
\makeatletter
\def\maxwidth{\ifdim\Gin@nat@width>\linewidth\linewidth\else\Gin@nat@width\fi}
\makeatother
\setkeys{Gin}{width=\maxwidth,keepaspectratio}

\usepackage[margin=2.5cm]{geometry}
\usepackage{setspace}
\usepackage{enumitem}
\setlist{itemsep=0.25\baselineskip, parsep=0pt, topsep=0.25\baselineskip, partopsep=0pt, leftmargin=*}

\usepackage{tocloft}
\usepackage{amsmath,amssymb,mathtools,amsthm}
\usepackage{longtable}

\swapnumbers

\theoremstyle{plain}
\newtheorem{theorem}{Theorem}
\newtheorem{lemma}[theorem]{Lemma}
\newtheorem{proposition}[theorem]{Proposition}

\theoremstyle{definition}
\newtheorem{definition}[theorem]{Definition}

\newtheorem{remark}[theorem]{Remark}

\usepackage[numbers,sort&compress]{natbib}

\usepackage{url}
\usepackage{hyperref}
\hypersetup{hidelinks}
\usepackage{authblk}
\usepackage{orcidlink}

\newcommand{\IR}{\mathbb{R}} % Real numbers #listed
\newcommand{\IN}{\mathbb{N}} % Natural numbers #listed
\newcommand{\del}{\partial} % Partial derivative #listed
\newcommand{\set}[1]{\{#1\}} % Simple set braces #listed
\newcommand{\Set}[2]{\left\{\, #1 \;\vert\; #2 \,\right\}} % Set-builder notation #listed
\newcommand{\ra}{\rightarrow} % Right arrow #listed
\newcommand{\tensor}{\otimes} % Tensor product #listed
\newcommand{\KL}{\mathcal{L}}

\newcommand{\qtext}[1]{\quad\text{#1}\quad} % Quad spaced text #listed
\newcommand{\ssum}[1]{\sum_{\substack{#1}}} % Sum with substacked condition #listed
\newcommand{\eps}{\varepsilon} % Epsilon variant #listed
\newcommand{\union}{\cup} % Union #listed
\newcommand{\tightlist}{\setlength{\itemsep}{0pt}\setlength{\parskip}{0pt}}
\renewcommand{\part}{\vdash} % Turnstile #listed (overrides LaTeX \part sectioning)
\newcommand{\opart}{\models} % Semantic entailment #listed
\DeclareMathOperator{\Pol}{Pol} % Polarization

\title{Faà di Bruno is Taylor Composition}
\author{Heinrich Hartmann\footnote{Hartmann IT GmbH / \href{https://heinrichhartmann.com/math}{heinrichhartmann.com/math} / \texttt{Heinrich@HeinrichHartmann.com}}\;\,\orcidlink{0000-0002-3929-2421}}
\date{}

\begin{document}

\maketitle
\vspace{-3em}\begin{center}\href{https://heinrichhartmann.com/math/2026-Faa-Di-Bruno}{heinrichhartmann.com/math/2026-Faa-Di-Bruno}\\[0.3em]\href{https://doi.org/10.48550/arXiv.2606.26133}{DOI} $\cdot$ \href{https://arxiv.org/abs/2606.26133}{arXiv} $\cdot$ \href{https://doi.org/10.5281/zenodo.18203201}{Zenodo}\end{center}\vspace{0.5em}

\begin{abstract}
We approach Faà di Bruno as a composition theorem for Taylor polynomials.
For $C^k$ maps $\phi: E \to F$ and $\psi: F \to G$ between Banach spaces,
let $T^k_\ast(\phi; x)$ denote the reduced Taylor polynomial of $\phi$ at $x$, obtained by removing the constant term.
We show that

$$T^k_\ast(\psi \circ \phi; x) = \pi_{\le k}\bigl(T^k_\ast(\psi; \phi(x)) \circ T^k_\ast(\phi; x)\bigr).$$

The proof is an elementary estimate of the Peano remainder and does not use partitions or combinatorial enumeration.

Expanding this composition identity recovers the classical Faà di Bruno formulas.
Polarization gives the multivariate partition formula (Lévy 2006), while coefficient extraction gives the multi-index formula (Constantine and Savits 1996).
Our approach separates the functorial nature of Taylor approximation from the combinatorial bookkeeping of polarization and coefficient extraction.

As an application, we give a general higher-order product rule.
\end{abstract}

\tableofcontents

\section{Introduction}\label{introduction}

Faà di Bruno formulas describe the higher-order behavior of derivatives
under composition. In one variable, the formula expresses the \(n\)-th
derivative of \(f \circ g\) as a universal polynomial in the derivatives
of \(f\) and \(g\). For \(f, g: \IR \to \IR\), this formula was first
derived in 1855 by Faà di Bruno \citep{FaaDiBruno1855}: \[
\begin{aligned}
\frac{d^n}{dx^n}\bigl(f\circ g\bigr)(x)
&=
\sum_{k=1}^n f^{(k)}(g(x))
\ssum{(\ast)}
\frac{n!}{m_1!\cdots m_n!}
\prod_{j=1}^n \Bigl(\frac{g^{(j)}(x)}{j!}\Bigr)^{m_j},
\end{aligned}
\] where \((\ast)\) runs over all tuples \(m_1,\dots,m_n\in\IN_0\) with
\(m_1+\dots+m_n=k\) and \(m_1+2m_2+\dots+nm_n=n\). The multivariate
generalization to maps \(\phi: \IR^d \to \IR^n\), \(f: \IR^n \to \IR\)
was established by Constantine and Savits \citep{CS1996}, who gave a
complete explicit formula for \(\del^\alpha(f \circ \phi)\) in full
multi-index form: \[
  \frac{\del^\alpha(f \circ \phi)(x)}{\alpha!}
  \;=\;
  \ssum{\beta \in \IN_0^n \\\\ 1\le|\beta|\le|\alpha|}
  \del^\beta f(y)\;
  \sum_{s\ge 1}\;
  \sum_{(\ast)}\;
  \prod_{r=1}^s
  \frac{1}{\beta_r!}
  \frac{\big( \del^{\alpha_r}\phi(x) \big)^{\beta_r}}
      {\alpha_r!^{|\beta_r|}},
\] where \((\ast)\) runs over tuples
\((\alpha_1,\dots,\alpha_s; \beta_1,\dots,\beta_s)\) with
\(\alpha_j \in \IN_0^d \setminus \set{0}\),
\(\beta_j \in \IN_0^n \setminus \set{0}\),
\(\alpha_1 < \cdots < \alpha_s\), subject to
\(\sum_r \alpha_r |\beta_r| = \alpha\) and \(\sum_r \beta_r = \beta\).

We give a functorial repackaging of these formulas, following the
philosophy that differentiable functions are functions well approximated
by polynomials, and the composition of differentiable functions should
be well approximated by the composition of their Taylor polynomials. Our
main result makes this precise:

\begin{theorem}[Taylor Composition \ref{thm:fdb-comp}]

Let \(E, F, G\) be Banach spaces, \(U_E \subset E\) and
\(U_F \subset F\) open. Let \(\phi \in C^k(U_E, F)\) and
\(\psi \in C^k(U_F, G)\) with \(y = \phi(x) \in U_F\). Then
\(\psi \circ \phi \in C^k(U_E, G)\) and the reduced Taylor polynomials
compose: \[
  T_{\ast}^k(\psi\circ\phi;\, x) = \pi_{\le k}\bigl(T_{\ast}^k(\psi;\, y) \circ T_{\ast}^k(\phi;\, x)\bigr).
\]

\end{theorem}

The proof starts directly from the Peano remainder form of the Taylor
approximation and establishes the Peano bounds for the composition. This
is a straightforward argument that requires neither heavy combinatorics,
multi-index calculations, nor partition arguments.

The functoriality of Taylor approximation under composition is certainly
well known to experts and mentioned in passing, e.g.~Malgrange
(\citep{malgrange}, p.16). However, we are not aware of a self-contained
treatment in this generality. Hernández Encinas and Muñoz Masqué
\citep{EM2003} are close in spirit, but work in the finite-dimensional
smooth setting and employ coordinate algebra on truncated local rings,
instead of bounding the Taylor residual directly.

From this composition principle we derive independent proofs of both
combinatorial forms. The partition form (Theorem \ref{thm:fdb-part},
\citep{Levy2006}) is obtained by polarization, extracting symmetric
multilinear forms from homogeneous polynomials: \[
  D(\psi\circ\phi;\, x;\, u_1,\dots,u_k)
  =
  \sum_{\pi \part [k]}
  D\bigl(\psi;\, y;\, (D(\phi;\, x;\, u_B))_{B\in\pi}\bigr).
\] The multi-index form (Theorem \ref{thm:fdb-multi}, \citep{CS1996}) is
obtained by monomial coefficient extraction from the composed Taylor
polynomials, recovering the formula above.

In summary, this note gives a short, self-contained account in the
general Banach space setting that presents all three forms of Faà di
Bruno in their natural relationship: Taylor composition as the analytic
core, the partition formula as its polarization, and the multi-index
formula as its coefficient extraction. As an application we derive a
higher-order product rule in both partition and multi-index form.

\begin{remark}[Notation]

In the following we will use the following notational conventions. We
generally prefer explicit arguments over subscript notation, writing
e.g.~\(\Delta(f; x; v)\) instead of \(\Delta_v f(x)\). Semicolons and
commas are both argument separators; we use semicolons as a visual hint
when arguments are of different kinds, as in
\(D(\phi; x; v_1, \dots, v_k)\). Dropping trailing arguments denotes
currying: for a function \(f(x; y; z)\), we write \(f(x; y)\) for the
map \(z \mapsto f(x; y; z)\). Thus \(D(\phi; x; u_1, \dots, u_k)\)
denotes the \(k\)-th derivative evaluated on directions; dropping the
vector arguments, \(D(\phi; x)\) denotes the corresponding multilinear
map when \(k\) is clear from context. Similarly \(T_{\ast}^k(\phi; x)\)
is a polynomial map and \(T_{\ast}^k(\phi; x; v)\) is its evaluation at
\(v\).

\end{remark}

\section{Polynomials in Banach
Spaces}\label{polynomials-in-banach-spaces}

In this section we develop the basic theory of polynomials in the
setting of Banach spaces, as found e.g.~in Dineen \citep{Dineen1999}.
The reader who is only interested in the finite-dimensional setting can
safely skip this section.

Let \(E, F\) be Banach spaces. We write \(\KL(E, F)\) for the Banach
space of bounded linear maps \(E \to F\) with operator norm
\(\|A\| = \sup_{\|v\| \leq 1} \|A(v)\|\), and more generally
\(\KL(^k E, F)\) for the Banach space of bounded \(k\)-fold multilinear
maps \(E^{\times k} \to F\) with norm
\(\|A_k\| = \sup_{\|v_i\| \leq 1} \|A_k(v_1,\dots,v_k)\|\). We write
\(\KL^s(^k E, F) \subset \KL(^k E, F)\) for the subspace of symmetric
multilinear maps.

\begin{definition}[Polynomials]

A map \(p: E\to F\) is a \emph{\(k\)-homogeneous polynomial} if there
exists \(A \in \KL^s(^k E,F)\) such that \[
  p(x)=\frac{1}{k!}\,A[x,\dots,x]=\frac{1}{k!}\,A[x^{\tensor k}], \qquad x\in E.
\] We write \(\mathcal{P}^k(E,F)\) for the space of such maps. A
\emph{polynomial map} \(p: E \to F\) is a finite sum of homogeneous
polynomials \(p = \sum_{k=0}^n p_k\) with
\(p_k \in \mathcal{P}^k(E,F)\). We write \(\mathcal{P}(E,F)\) for the
space of all polynomial maps, \(\mathcal{P}_{\le n}(E,F)\) for those of
degree \(\le n\), and \(\mathcal{P}(E) := \mathcal{P}(E,\IR)\).

\end{definition}

\begin{definition}[Forward differences]

For \(f: E \to F\) and \(x, v \in E\), define the \emph{forward
difference} \[
  \Delta(f; x; v) := f(x + v) - f(x).
\] The \emph{iterated forward difference} is defined recursively:
\(\Delta(f; x) := f(x)\) for \(k = 0\), and \[
  \Delta(f; x; v_1, \dots, v_k) := \Delta(w \mapsto \Delta(f; x; v_2, \dots, v_k);\, x;\, v_1)
\] for \(k \geq 1\). Since the operators \(\Delta_{v_i}\) commute, this
is symmetric in \(v_1, \dots, v_k\). Expanding the recursion gives the
alternating sum over vertices of the \(k\)-cube: \[
  \Delta(f; x; v_1, \dots, v_k) = \sum_{S \subseteq [k]} (-1)^{k - |S|}\, f(x + \textstyle\sum_{i \in S} v_i).
\]

\end{definition}

\begin{lemma}[Polarization]

\label{lem:polarization} Let \(p: E \to F\) be a polynomial map,
represented as \(p(x) = \sum_{k=0}^n \frac{1}{k!} A_k[x^{\tensor k}]\)
with \(A_k \in \KL^s(^k E, F)\). Then the top-degree multi-linear form
is recovered as a forward difference: \[
  A_n[v_1,\dots,v_n] = \Delta(p; 0; v_1, \dots, v_n).
\] Furthermore, every polynomial map \(p\) has a unique representation
\(p(x) = \sum_{k=0}^n \frac{1}{k!} A_k[x^{\tensor k}]\) with
\(A_k \in \KL^s(^k E, F)\).

\end{lemma}

\begin{proof}

The polarization formula is Möbius inversion on the Boolean lattice
\(2^{[n]}\): By multilinearity,

\[
p(\sum_{i \in I} v_i) 
= \sum_{k=0}^n \sum_{i_1,\dots,i_k \in I}  \frac{1}{k!} A_k[v_{i_1},\dots,v_{i_k}] 
= \sum_{J \subset I} A_{|J|}[v_J]
\]

where \(A_k[v_J]\) denotes evaluation on the \(|J|\)-tuple indexed by
\(J\). The alternating sum \(\sum_{I \subset [n]} (-1)^{n-|I|}\) inverts
this: all terms with \(J \subsetneq [n]\) cancel by the identity
\(\sum_{I \supset J} (-1)^{n-|I|} = (1-1)^{n-|J|} = 0\), leaving only
\(J = [n]\). Uniqueness follows by induction: recover \(A_n\) via
polarization, subtract, and repeat.

\end{proof}

\begin{definition}[Degree Truncation]

Let \(p: E \ra F\) be a polynomial map. By Lemma \ref{lem:polarization}
the representation
\(p(x) = \sum_{k=0}^n  \frac{1}{k!} A_k[x^{\tensor k}]\) is unique. We
can hence define:

\begin{itemize}
\tightlist
\item
  \(\deg(p):= n \in \IN_0\) the degree of \(p\). By convention
  \(\deg(0) = - \infty\).
\item
  \(\pi_k(p) := p_k: E \ra F\) the degree-k part of \(p\).
\item
  \(\Pol_k(p) := A_k \in \KL^s(^k E, F)\) the degree-k polarization of
  \(p\).
\item
  \(\pi_{\leq k}(p) \in \mathcal{P}_{\le k}(E,F)\) the (lower) degree
  truncation of \(p\).
\item
  \(\pi_{> k}(p) \in \mathcal{P}_{> k}(E,F)\) the upper degree
  truncation of \(p\).
\end{itemize}

\end{definition}

We will make use of the following properties of polynomial maps:

\begin{lemma}[Polynomial Properties]

\label{lem:poly-props} Let \(p : E \ra F\), \(q: F \ra G\) be polynomial
maps.

\begin{enumerate}
\def\labelenumi{\arabic{enumi}.}
\tightlist
\item
  (Composition) The composition \(q \circ p: E \ra G\) is a polynomial
  with \(\deg(q \circ p) \leq \deg(q) \cdot \deg(p)\).
\item
  (Lipschitz) There exists \(\eps > 0, C>0\) such that
  \(\|\Delta(p; 0; x)\| \le C \|x\|\) for all \(\|x\| < \eps\).
\item
  (Upper Vanishing) We have \(\| \pi_{> k}(p)(x)\| / \|x\|^{k} \ra 0\)
  for \(x \ra 0\).
\item
  (Lower Vanishing) If \(\|p(x)\| / \|x\|^{k} \ra 0\) for \(x \ra 0\),
  and \(k \geq \deg(p)\) then \(p = 0\).
\end{enumerate}

\end{lemma}

\begin{proof}

Ad 1) Expanding \(q(p(x))\) by multilinearity yields a sum of terms
\(B[A_1[x^{\tensor j_1}], \dots, A_\ell[x^{\tensor j_\ell}]]\) where
\(B \in \KL(^\ell F, G)\) and \(A_i \in \KL(^{j_i} E, F)\) are bounded
multilinear.

The composition
\((u_1,\dots,u_k) \mapsto B[A_1[u_{I_1}], \dots, A_\ell[u_{I_\ell}]]\)
is bounded multilinear, hence each term is a homogeneous polynomial of
degree \(j_1 + \cdots + j_\ell \leq \ell \cdot \deg(p)\). Thus
\(q \circ p\) is a polynomial with
\(\deg(q \circ p) \leq \deg(q) \cdot \deg(p)\).

Ad 2) Write
\(\Delta(p; 0; x) = p(x) - p(0) = \sum_{j=1}^m \frac{1}{j!} A_j[x^{\tensor j}]\)
with \(A_j \in \KL^s(^j E, F)\) bounded. For \(\|x\| < 1\) we have
\(\|x\|^j \leq \|x\|\), hence
\(\|\Delta(p; 0; x)\| \leq \sum_{j=1}^m \frac{\|A_j\|}{j!} \|x\|^j \leq C\|x\|\)
with \(C = \sum_{j=1}^m \frac{\|A_j\|}{j!}\).

Ad 3) Each homogeneous component \(p_j\) of degree \(j > k\) satisfies
\(\|p_j(x)\| \leq C_j \|x\|^j\), hence
\(\|p_j(x)\|/\|x\|^k = C_j \|x\|^{j-k} \to 0\) as \(x \to 0\).

Ad 4) If \(p \neq 0\), let
\(j_0 := \min\{j : p_j \neq 0\} \leq \deg(p) \leq k\) and choose
\(v \in E\) with \(p_{j_0}(v) \neq 0\). Then
\(p(tv) = t^{j_0} p_{j_0}(v) + O(t^{j_0+1})\). If \(j_0 < k\), then
\(\|p(tv)\|/t^k\) diverges as \(t \to 0^+\). If \(j_0 = k\), then
\(\|p(tv)\|/t^k \to \|p_{j_0}(v)\| \neq 0\). In either case
\(p = o(\|x\|^k)\) is impossible.

\end{proof}

\section{Fréchet Differentiability}\label{fruxe9chet-differentiability}

We summarize the classical notion of Fréchet differentiability following
\citep{LangRFA}.

\begin{definition}[Fréchet differentiability \citep{LangRFA} XIII.6]

Let \(E,F\) be Banach spaces, \(U_E \subset E\) an open subset and
\(\phi:U_E \ra F\) a map.

\begin{itemize}
\item
  (Fréchet differentiable) We say that \(\phi\) is Fréchet
  differentiable at a point \(x \in U_E\) if there exists a bounded
  linear map \(L \in \KL(E,F)\) such that \[
      \lim_{v\to 0} \frac{1}{\|v\|}\, \|\Delta(\phi; x; v) - L(v)\| = 0.
    \] The map \(L\) is uniquely determined and written \(D(\phi; x)\),
  so that \(D(\phi; x; v) = L(v)\).
\item
  (\(C^1\) maps) We say that \(\phi\) is continuously Fréchet
  differentiable (\(C^1\)) on \(U_E\) if \(D(\phi; x)\) exists for every
  \(x \in U_E\) and the map \(x \mapsto D(\phi; x) \in \KL(E,F)\) is
  continuous.
\item
  (\(C^k\) maps) Recursively, \(\phi\) is \(k\)-times continuously
  Fréchet differentiable (\(C^k\)) if \(\phi\) is \(C^1\) and
  \(x \mapsto D(\phi; x)\) is \(C^{k-1}\). We write \(C^k(U_E, F)\) for
  the space of such maps, and denote by \[
      D^k(\phi; x; v_1,\dots,v_k) := D(z \mapsto D^{k-1}(\phi; z; v_1,\dots,v_{k-1});\, x;\, v_k)
    \] the \(k\)-th Fréchet differential. This is a symmetric
  \(k\)-linear form: \(D^k(\phi; x) \in \KL^s(^k E, F)\). When the order
  is clear from context, we write \(D(\phi; x; v_1,\dots,v_k)\) for
  \(D^k(\phi; x; v_1,\dots,v_k)\).
\end{itemize}

\end{definition}

\begin{definition}[Taylor polynomials]

\label{def:taylor} Let \(\phi \in C^k(U_E, F)\) and \(x \in U_E\). The
\emph{Taylor polynomial} of \(\phi\) at \(x\) is \[
  T^k(\phi; x; v) := \sum_{\ell=0}^k \frac{1}{\ell!} D^\ell(\phi; x; v,\dots,v), \quad T^k(\phi; x) \in \mathcal{P}_{\le k}(E, F).
\] The \emph{reduced Taylor polynomial} is
\(T_{\ast}^k(\phi; x; v) := T^k(\phi; x; v) - \phi(x)\), so that
\(T_{\ast}^k(\phi; x; 0) = 0\).

\end{definition}

\begin{remark}

The full Taylor polynomial \(T^k\) is the correct object for studying
the commutative algebra of functions under pointwise multiplication. The
reduced Taylor polynomial \(T_{\ast}^k\) is the correct object for
studying functions under composition. We will focus on the reduced
version in what follows, returning to the full version for the Leibniz
rule (\ref{prop:leibniz}).

\end{remark}

\begin{proposition}[Taylor Approximation \citep{LangRFA}, XIII.6]

\label{prop:taylor} Let \(\phi \in C^k(U_E, F)\) and \(x \in U_E\).

\begin{enumerate}
\def\labelenumi{\arabic{enumi}.}
\item
  The Taylor remainder in Peano form
  \(R^k(\phi; x) := \Delta(\phi; x) - T_{\ast}^k(\phi; x)\) satisfies:
  \[
   \Delta(\phi; x; v) = T_{\ast}^k(\phi; x; v) + R^k(\phi; x; v)
   \] and \(\|R^k(\phi; x; v)\|/\|v\|^k \to 0\) as \(v \to 0\).
\item
  (Uniqueness) Let \(\Delta(\phi; x) = T + R\) be any other
  decomposition with \(T \in \mathcal{P}_{\leq k}(E,F)\), \(T(0) = 0\),
  and \(\|R(v)\|/\|v\|^k \ra 0\) as \(v \ra 0\). Then
  \(T = T_{\ast}^k(\phi; x)\) and \(R = R^k(\phi; x)\).
\end{enumerate}

\end{proposition}

\begin{proof}

Part 1 is standard (\citep{LangRFA}, XIII.6). For the uniqueness (part
2), assume \(\Delta(\phi; x) = T + R = T' + R'\) are two such
decompositions. Then \(p = T-T' = R-R'\) is a polynomial of degree
\(\leq k\) with \(\|p(v)\|/\|v\|^k \ra 0\) for \(v \ra 0\). By Lemma
\ref{lem:poly-props} we must have \(p = 0\).

\end{proof}

\begin{lemma}[Lipschitz Continuity]

\label{lem:lipschitz} If \(\phi \in C^k(U_E, F)\) with \(k \geq 1\),
then \(\phi\) is locally Lipschitz continuous: for every \(x \in U_E\)
there exist \(C > 0, \eps > 0\) such that
\(\| \Delta(\phi; x; v) \| \leq C \|v\|\) for all \(\|v\| < \eps\).

\end{lemma}

\begin{proof}

We may assume \(x = 0\). Write \(T_{\ast}^k(\phi; 0) = P\) and
\(R^k(\phi; 0) = R\). By Lemma \ref{lem:poly-props} there exist
\(\eps > 0\) and \(C_1 > 0\) such that \(\|P(v)\| \leq C_1\|v\|\) for
\(\|v\| < \eps\). Since \(R(v) = o(\|v\|^k)\) we have
\(\|R(v)\| \leq C_2\|v\|^k \leq C_2\|v\|\) for
\(\|v\| < \min(\eps, 1)\). Thus
\(\|\Delta(\phi; 0; v)\| \leq \|P(v)\| + \|R(v)\| \leq (C_1 + C_2)\|v\|\)
for \(\|v\|\) small.

\end{proof}

\section{Taylor Composition}\label{taylor-composition}

We now prove the main result: reduced Taylor polynomials compose, up to
degree truncation.

\begin{theorem}[Taylor Composition]

\label{thm:fdb-comp} Let \(E,F,G\) be Banach spaces. Let
\(U_E \subset E\) be open, \(x \in U_E\), and \(\phi \in C^k(U_E, F)\).
Let \(U_F \subset F\) be open with \(y = \phi(x) \in U_F\) and
\(\psi \in C^k(U_F, G)\).

Then \(\psi \circ \phi \in C^k(U_E, G)\) and the reduced Taylor
polynomials compose: \[
    T_{\ast}^k(\psi \circ \phi;\, x) = \pi_{\leq k}\bigl( T_{\ast}^k(\psi;\, y) \circ T_{\ast}^k(\phi;\, x) \bigr).
    \]

\end{theorem}

\begin{proof}

By translation, we may assume \(x=0\) and \(\phi(x) = y = 0\) as well as
\(\psi(y) = 0\) throughout the proof. In particular
\(\Delta(\phi; 0) = \phi\), \(\Delta(\psi; 0) = \psi\).

Choose Taylor decompositions \(\phi = P + R\) and \(\psi = Q + S\), with
polynomials \(P,Q\) of degree \(\le k\) and remainders \(R,S\) vanishing
to order \(k\) at \(0\). Now write
\(\psi\circ\phi = Q(\phi) + S(\phi)=\pi_{\le k}(Q\circ P) + E\), with:
\[
  E := Q(\phi) + S(\phi) - Q(P) + \pi_{> k}(Q\circ P).
\] We have to show that \(\|E(x)\| / \|x\|^k \to 0\) for \(x \to 0\).

Clearly \(\| \pi_{> k}(Q\circ P) \| / \|x\|^k \to 0\) for \(x \to 0\) by
Lemma \ref{lem:poly-props}.

To see that \(\|S(\phi(x))\|/\|x\|^k \to 0\) we argue as follows: By the
Lipschitz property of \(\phi\) (Lemma \ref{lem:lipschitz}), we find
\(\delta > 0, C > 0\) so that \(\|\phi(x)\| < C \|x\|\) for all
\(\|x\| < \delta\). Since \(S\) is a Taylor residual we have
\(\|S(y)\| < \eta(y) \|y\|^k\) with \(\eta(y) \to 0\) as \(y \to 0\).
Hence \(\|S(\phi(x))\|/\|x\|^k \leq \eta(\phi(x)) C^k \to 0\) as
\(x \to 0\).

For the term \(Q(P+R) - Q(P)\), we decompose
\(Q(y) = \sum_{\ell=0}^k Q_\ell(y) = \sum_\ell \frac{1}{\ell!} B_\ell[y^{\tensor \ell}]\)
where \(Q_\ell = \pi_\ell Q\) and \(B_\ell = \Pol_\ell Q\). Then \[
  Q_\ell(P+R) - Q_\ell(P)
  = \sum_{\substack{i+j=\ell\\ j\ge1}} \frac{1}{i!\,j!}\,B_\ell\big(P^{\tensor i}, R^{\tensor j}\big).
\]

It suffices to show that each summand
\(B_\ell\big(P^{\tensor i}, R^{\tensor j})\) is in \(o(\|x\|^k)\). To
this end we note that:

\begin{itemize}
\tightlist
\item
  \(\|B_\ell(P(x)^{\tensor i}, R(x)^{\tensor j})\| \leq \|B_\ell\| \cdot \|P(x)\|^i \|R(x)\|^j\)
  as \(B_\ell \in \KL^s(^\ell F,G)\) bounded.
\item
  \(\|P(x)\| \le C\|x\|\) for \(\|x\| < \delta\) as \(P\) is Lipschitz
  (Lemma \ref{lem:poly-props}) and \(P(0) = 0\).
\item
  \(\|R(x)\| < \eta(x) \|x\|^k\) with \(\eta(x) \ra 0\) for \(x \to 0\)
  as \(R\) is a Taylor remainder.
\end{itemize}

So that: \[
  \|B_\ell(P^{\tensor i}(x),R^{\tensor j}(x))\| / \|x\|^k
  \leq
  \|B_\ell\| \; C^i \; \eta(x)^j \; \|x\|^{i + (j - 1) k}
  \to 0 \qtext{as} x \to 0
\] For the last step, note that \(j \geq 1\) so the factor
\(\eta(x)^j \to 0\) as \(x \to 0\), and \(i + (j - 1) k \geq 0\) hence
\(\|x\|^{i + (j - 1) k}\) stays bounded (and even goes to zero for
\(j \geq 2\)).

\end{proof}

\begin{remark}[Pointwise Taylor approximation]

The assumptions of this theorem can be slightly weakened. The proof does
not use the existence of derivatives in a neighborhood or the continuity
of \(z \mapsto D^j(\phi; z)\). What is needed is only a decomposition
\(\Delta(\phi; x) = P + R\) with \(P \in \mathcal{P}_{\le k}(E,F)\),
\(P(0) = 0\), and \(\|R(v)\|/\|v\|^k \to 0\) as \(v \to 0\) (the Peano
remainder condition), and similarly for \(\psi\) at \(y\). The theorem
therefore applies to a broader class of functions that are well
approximated by a polynomial at a single point.

\end{remark}

\section{Partition Form Faà di
Bruno}\label{partition-form-fauxe0-di-bruno}

The partition form of Faà di Bruno is obtained by polarizing the Taylor
composition theorem: extracting the degree-\(k\) symmetric multilinear
part from both sides.

\begin{definition}

Let \(I\) be a finite set.

\begin{itemize}
\item
  An un-ordered (non-empty) partition \(\pi\) of \(I\) is a set
  \(\pi = \set{B_1,\dots,B_r}\) of subsets \(B_i \subset I\)
  (``blocks''), so that \(B_i\) are disjoint, \(I = \union_{i=1}^r B_i\)
  and \(B_i \neq \emptyset\). We write \(\pi \part I\) in this case and
  call \(|\pi| = r\) the rank of \(\pi\).
\item
  An \emph{ordered \(r\)-decomposition} of \(I\) is a tuple
  \((B_1,\dots,B_r)\) of subsets \(B_i \subset I\), so that \(B_i\) are
  disjoint and \(I = \union_{i=1}^r B_i\). We write
  \((B_1,\dots,B_r) \opart I\) in this case.

  Note that we allow empty blocks, so that this datum is equivalent to a
  map \(\sigma: I \ra [r]\) with \(B_i = \sigma^{-1}\set{i}\).
\item
  For a set \(X\), a tuple \(v = (v_i)_{i \in I} \in X^I\) and a block
  \(B \subset I\), we write \(v_B := (v_b)_{b \in B} \in X^B\) for the
  restricted tuple.
\end{itemize}

\end{definition}

\begin{theorem}[Multivariate Faà di Bruno - Partition Form, cf. \citep{Levy2006}]

\label{thm:fdb-part} Let \(x \in U_E \subset E\) and set
\(y := \phi(x) \in F\). Assume that \(\phi \in C^k(U_E, F)\) and
\(\psi \in C^k(U_F, G)\). For \(u_1,\dots,u_k\in E\) we have: \[
  D(\psi\circ\phi;\, x;\, u_1,\dots,u_k)
  =
  \sum_{\pi \part [k]}
  D\bigl(\psi;\, y;\, (D(\phi;\, x;\, u_B))_{B\in\pi}\bigr).
\] The order of the blocks is immaterial because higher derivatives are
symmetric.

\end{theorem}

\begin{proof}

Fix \(u_1,\dots,u_k\in E\). By translation, we may again assume \(x=0\)
and \(y=0\) throughout the proof. Fix Taylor decompositions
\(\phi = P + R\), \(\psi = Q + S\) with \(P,Q\) polynomial of degree
\(\le k\). By \ref{thm:fdb-comp} we have
\(T_{\ast}^k(\psi \circ \phi; 0) = \pi_{\leq k}(Q \circ P)\), and hence
\(D(\psi\circ\phi; 0) = \Pol_k (Q \circ P)\).

Write the polynomial maps in polarized form as \[
  P(x) = \sum_{m=1}^k \frac{1}{m!}\,P_m[x^{\tensor m}], \qquad
  Q(y) = \sum_{r=1}^k \frac{1}{r!}\,Q_r[y^{\tensor r}].
\]

Expanding \(Q_r\) by multilinearity and isolating the degree \(k\) term
gives:

\[
  \pi_k Q(P(x))
  = 
  \sum_{r=1}^k
  \sum_{\substack{m_1,\dots,m_r\ge1 \\\\ m_1+\cdots+m_r = k}}
  \frac{1}{r!} Q_r\bigl[\frac{1}{m_1!} P_{m_1}[x^{\tensor m_1}],\dots,\frac{1}{m_r!} P_{m_r}[x^{\tensor m_r}]\bigr].
\]

Applying \(\Pol_k\) on both sides using Lemma \ref{lem:pol-comp} below
gives:

\[
  D(\psi\circ\phi; 0; u_1,\dots,u_k)
  = 
  \sum_{r=1}^k
  \ssum{m_1,\dots,m_r\ge1 \\\\ m_1+\cdots+m_r = k}
  \ssum{(B_1,\dots,B_r) \opart [k] \\\\ |B_i| = m_i}
  \frac{1}{r!}
  Q_r\bigl[P_{m_1}[u_{B_1}],\dots,P_{m_r}[u_{B_r}]\bigr].
\]

Now observe that there are exactly \(r!\) re-orderings of each ordered
decomposition \((B_1,\dots,B_r)\), and these re-orderings yield the same
summand as \(Q_r\) is symmetric. Hence the sum over
\(m_i, B \opart [k]\) can be replaced by a single sum over un-ordered
partitions \(\pi \part [k]\) with given block sizes \(|B_i| = m_i\).
Furthermore we have \(Q_r[v_1,\dots,v_r] = D(\psi; 0; v_1,\dots,v_r)\)
and \(P_m[u_1,\dots,u_m] = D(\phi; 0; u_1,\dots,u_m)\) by definition.
Substituting these into the above formulas yields the claim.

\end{proof}

It remains to prove the composition polarization lemma used above.

\begin{lemma}[Composition Polarization]

\label{lem:pol-comp} If \(k = m_1 + \dots + m_r \in \IN_0\) and
\(A_i \in \KL^s(^{m_i} E,F)\), and \(B \in \KL^s(^r F,G)\). Let \[
H(x) = B[\frac{1}{m_1!}A_1[x^{\tensor m_1}], \dots, \frac{1}{m_r!}A_r[x^{\tensor m_r}]]
\] then: \[
\Pol_k(H)[u_1,\dots,u_k] = \ssum{(B_1,\dots,B_r) \opart [k], \\\\ |B_i| = m_i} B[A_1[u_{B_1}], \dots, A_r[u_{B_r}]].
\]

\end{lemma}

\begin{proof}

Let \(G[u_1,\dots,u_k]\) be the claimed expression for
\(\Pol_k(H)[u_1,\dots,u_k]\). As both sides are symmetric, bounded
\(k\)-multilinear forms in \(u_1,\dots,u_k\) we only have to validate
that \(\frac{1}{k!} G[x^{\tensor k}] = H\). We find: \[
G[x^{\tensor k}] 
= \sum_{(B_1,\dots,B_r) \opart [k], \\\\ |B_i| = m_i} B[A_1[x^{\tensor m_1}], \dots, A_r[x^{\tensor m_r}]] 
= \frac{k!}{m_1! \dots m_r!} \cdot B[A_1[x^{\tensor m_1}], \dots, A_r[x^{\tensor m_r}]]
\] as there are \(k! / m_1! \dots m_r!\) ordered decompositions of
\([k]\) into subsets of sizes \(m_1,\dots,m_r\). Re-arranging the
factorials yields the claim.

\end{proof}

\section{Multi-index Faà di Bruno}\label{multi-index-fauxe0-di-bruno}

In coordinates, the Taylor composition theorem yields an explicit
formula for partial derivatives of a composition by monomial coefficient
extraction.

\begin{definition}[Multi-Indices]

\leavevmode

\begin{itemize}
\item
  Elements \(\nu = (\nu_1,\dots,\nu_d) \in \IN_0^d\) are called
  multi-indices of dimension \(d\). We write \(|\nu| = \sum_i \nu_i\)
  for the degree and \(\nu! := \nu_1! \cdots \nu_d!\) for the factorial.
\item
  If \(A\) is a commutative algebra and \(v \in A^d\) is a \(d\)-tuple,
  then we write: \(v^\nu = \prod_i v_i^{\nu_i} \in A\).
\item
  If \(E = \IR^d\), write \(v^1,\dots,v^d: E \ra \IR\) for the
  coordinate projections. The monomials \(v^\nu\), \(\nu \in \IN_0^d\),
  form a basis of \(\mathcal{P}(E)\).
\item
  If \(U \subset \IR^d\) is open and \(\phi \in C^k(U, F)\) with
  \(x \in U\). Let \(\nu \in \IN_0^d\) with \(k = |\nu|\). We define the
  partial derivative \(\del^\nu \phi (x) \in F\) via the basis expansion
  of
  \(T_{\ast}^k(\phi; x) \in \mathcal{P}(\IR^d,F) = \mathcal{P}(\IR^d) \tensor F\)
  as coefficient of \(v^\nu/\nu!\): \[
      T_{\ast}^k(\phi; x; v) = \sum_{1 \leq |\nu| \leq k} \frac{v^\nu}{\nu!} \; \del^\nu \phi(x).
    \]
\end{itemize}

\end{definition}

\begin{theorem}[Multivariate Faà di Bruno - Multi-index Form, cf. \citep{CS1996}]

\label{thm:fdb-multi} Let \(x \in U \subset \IR^d\) and set
\(y := \phi(x) \in \IR^n\). Assume that \(\phi \in C^k(U, \IR^n)\) and
\(f \in C^k(\IR^n, \IR)\).

For \(\alpha \in \IN_0^d\), \(\alpha \neq 0\) the partial derivatives of
the composition are given by:

\begin{align*}
  \frac{\del^\alpha (f \circ \phi) (x)}{\alpha!}
=
  \sum_{\beta\in\IN_0^n}
  \del^\beta f(y)\,
  \sum_{\rho\,(\ast)}
  \Bigl(
    \prod_{\gamma \in \IN_0^d}
    \frac{1}{\rho(\gamma)!}
    \frac{(\del^\gamma \phi(x))^{\rho(\gamma)}}{(\gamma!)^{|\rho(\gamma)|}}
  \Bigr)
\end{align*}

where \((\ast)\) runs over all finitely supported maps
\(\rho:\IN_0^d\setminus\set{0} \ra \IN_0^n\) satisfying
\(\sum_\gamma \rho(\gamma) = \beta\) and
\(\sum_\gamma \gamma|\rho(\gamma)| = \alpha\).

\end{theorem}

This is equivalent to the tuple form stated in the introduction:
enumerating the support of \(\rho\) in any total order on \(\IN_0^d\)
recovers the Constantine--Savits formula.

\begin{proof}

We may assume that \(x = y = f(y) = 0\), without loss of generality.

Step 1: Taylor Expansion) Write \(v^1,\dots,v^d\) for the coordinate
functions on \(\IR^d\) and \(w^1,\dots,w^n\) for the coordinates on
\(\IR^n\). Expand the Taylor polynomials of \(\phi\) and \(f\) as \[
  P(v) = T_{\ast}^k(\phi; 0; v) = \ssum{\alpha \in \IN_0^d \\\\ 1 \leq |\alpha| \leq k} \frac{v^\alpha}{\alpha!} \; p_\alpha,
  \qquad
  F(w) = T_{\ast}^k(f; 0; w) = \ssum{\beta \in \IN_0^n \\\\ 1\leq |\beta| \leq k} \frac{w^\beta}{\beta!} \; f_\beta,
\] with \(p_\alpha = \del^\alpha \phi(0) \in \IR^n\) and
\(f_\beta = \del^\beta f(0) \in \IR\).

Step 2: Polynomial Composition) By Theorem \ref{thm:fdb-comp} we have
\(\del^\nu (f \circ \phi)(0)/\nu! = [v^\nu](F \circ P)\). We compute
\([v^\nu](F \circ P)\) as follows: Let \(m \geq 0\). For each coordinate
\(w^i\) we have:

\[
(w^i\circ P)^{m}
=
  \Bigl(\sum_{\alpha \in \IN_0^d}
  \frac{p_\alpha^i}{\alpha!}\; v^\alpha\Bigr)^{m}
=
  \sum_{\alpha_1,\dots,\alpha_m \in \IN_0^d}
  \prod_{j=1}^m
  \frac{p_{\alpha_j}^i}{\alpha_j!} v^{\alpha_j}
=
  \ssum{\rho_i:\IN_0^d \ra \IN_0 \\\\ \sum_\alpha \rho_i(\alpha) = m}
  \Bigl(\frac{m!}{\prod_{\alpha \in \IN_0^d} \rho_i(\alpha)!} \Bigr)
  \prod_{\alpha\in \IN_0^d}
  \Bigl(\frac{p_\alpha^i}{\alpha!}\,v^\alpha\Bigr)^{\rho_i(\alpha)}.
\]

Note that there are only finitely many maps \(\rho:\IN_0^d \ra \IN_0\)
with \(\sum_\alpha \rho(\alpha) = n\), and that each of those maps is
finitely supported i.e.~\(\rho(\alpha) = 0\) except for finitely many
\(\alpha \in \IN_0^{d}\). In the second step the map \(\rho_i\) counts
the number of occurrences of \(\alpha\) in a given tuple
\((\alpha_1,\dots,\alpha_m)\),
i.e.~\(\rho_i(\alpha) = \# \Set{j}{\alpha_j = \alpha}\). For each given
map \(\rho_i\) we have \(m!/\prod_{\alpha \in \IN_0^d} \rho_i(\alpha)!\)
tuples which give the same map.

Now let \(\beta \in \IN_0^n\) and apply this formula to each
\(m = \beta_i \in \IN_0\). Write \(\rho: \IN_0^d \ra \IN_0^n\),
\(\rho(\alpha) = (\rho_1(\alpha),\dots,\rho_n(\alpha))\), so that:

\begin{align*}
  (w \circ P)^\beta
&=
  \prod_{i=1}^n (w^i \circ P)^{\beta_i}
=
  \ssum{\rho:\IN_0^d \ra \IN_0^n\\\\ \sum_\alpha \rho(\alpha) = \beta}
  \Bigl(
    \prod_{i=1}^n
    \frac{\beta_i!}{\prod_\alpha \rho_i(\alpha)!}
  \Bigr)
  \prod_{i=1}^n
  \prod_{\alpha \in \IN_0^d}
  \Bigl(
    \frac{p_\alpha^i}{\alpha!}\,v^\alpha
  \Bigr)^{\rho_i(\alpha)}
  \\\\
&=
  \beta! \;
  \ssum{\rho:\IN_0^d \ra \IN_0^n \\\\ \sum_\alpha \rho(\alpha) = \beta}
  \Bigl(
    \prod_{\alpha \in \IN_0^d}
    \frac{1}{\rho(\alpha)!}
    \frac{(p_\alpha)^{\rho(\alpha)}}{(\alpha!)^{|\rho(\alpha)|}}
  \Bigr) \;
  v^{\sum_\alpha \alpha|\rho(\alpha)|}.
\end{align*}

Substituting into \(F(P(v))\) and extracting the \(v^\alpha\)
coefficient gives:

\begin{align*}
  \frac{\del^\alpha (f \circ \phi) (0)}{\alpha!}
= 
  [v^\alpha]\, (F \circ P)
=
  [v^\alpha]\,
  \sum_{\beta\in\IN_0^n}
  \frac{f_\beta}{\beta!}\,(w \circ P)^\beta
=
  \sum_{\beta\in\IN_0^n}
  f_\beta\,
  \sum_{\rho\,(*)}
  \Bigl(
    \prod_{\gamma \in \IN_0^d}
    \frac{1}{\rho(\gamma)!}
    \frac{(p_\gamma)^{\rho(\gamma)}}{(\gamma!)^{|\rho(\gamma)|}}
  \Bigr)
\end{align*}

Here \((\ast)\) runs over all maps \(\rho:\IN_0^d \ra \IN_0^n\)
satisfying \(\sum_\gamma \rho(\gamma) = \beta\) and
\(\sum_\gamma \gamma|\rho(\gamma)| = \alpha\). Since \(P(0) = 0\), only
maps with \(\rho(0) = 0\) contribute, so we may restrict to
\(\rho:\IN_0^d\setminus\set{0} \ra \IN_0^n\). The constraint
\(\sum_\gamma \gamma|\rho(\gamma)| = \alpha\) with \(|\gamma| \geq 1\)
forces \(1 \leq |\beta| \leq |\alpha|\).

\end{proof}

\begin{remark}[Factorials and Symmetry]

Controlling the complexity of combinatorial expressions is a key
technical difficulty when working with Faà di Bruno. An important
insight in this presentation is that factorials are always associated
with symmetrization: passing from multilinear maps to polynomials via
diagonal evaluation, or quotienting out ordering information (e.g., from
ordered decompositions to unordered partitions). Formulas are often
simpler when symmetry is not yet imposed. We deliberately keep each
factorial adjacent to the corresponding symmetrized quantity to make the
combinatorics transparent.

\end{remark}

\section{Product Rule}\label{product-rule}

As an application of the Taylor composition theorem, we derive a
higher-order product rule for pointwise multiplication.

\begin{proposition}[Leibniz rule]

\label{prop:leibniz} Let \(E\) be a Banach space and let
\(f_1,\dots,f_n:E\to\IR\) be \(k\)-times Fréchet differentiable at a
point \(x \in E\). Recall the (unital) Taylor polynomial
\(T^k(f; x; v) = f(x) + T_{\ast}^k(f; x; v)\) from \ref{def:taylor}.
Then:

\begin{enumerate}
\def\labelenumi{\arabic{enumi}.}
\item
  (Taylor form). \[
     T^k(f_1 \cdots f_n;\, x;\, v) = \pi_{\leq k}\bigl(T^k(f_1;\, x;\, v) \cdots T^k(f_n;\, x;\, v)\bigr).
   \]
\item
  (Ordered-decomposition form.) For \(u_1,\dots,u_k\in E\), \[
  D(f_1 \cdots f_n;\, x;\, u_1,\dots,u_k)
  =
  \ssum{(B_1,\dots,B_n) \opart [k]}\ D(f_1;\, x;\, u_{B_1}) \cdots D(f_n;\, x;\, u_{B_n})
  \] where we use the convention \(D(f_i; x; u_\emptyset) = f_i(x)\) for
  the empty block.
\item
  (Multi-index form.) If \(E=\IR^d\) and \(\nu\in\IN_0^d\) with
  \(|\nu|=k\), then \[
  \frac{\del^\nu}{\nu!}(f_1 \cdots f_n)(x)
  =
  \ssum{\nu_1,\dots,\nu_n \in \IN_0^d \\\\ \nu_1 + \dots + \nu_n = \nu}\ \frac{\del^{\nu_1}f_1(x)}{\nu_1!} \cdots \frac{\del^{\nu_n}f_n(x)}{\nu_n!}
  \] where the sum ranges over all tuples
  \((\nu_1,\dots,\nu_n)\in(\IN_0^d)^n\) with \(\nu_i \in \IN_0^d\) with
  \(\nu_1+\cdots+\nu_n=\nu\).
\end{enumerate}

\end{proposition}

\begin{proof}

For \(k = 0\) both expressions are trivially true. Hence we may assume
\(k \geq 1\).

Ad 1) Set \(\phi = (f_1,\dots,f_n): E \ra \IR^n\), and
\(\psi(y) = \prod_{i=1}^n y_i: \IR^n \ra \IR\), \(c_i = f_i(x)\) so that
\(\psi \circ \phi = f_1\cdots f_n\) and \(c = \phi(x)\).

As \(\psi(y) = \prod_{i=1}^n y_i\) is a polynomial itself, the reduced
Taylor polynomial at \(c\) can be calculated by translation and degree
truncation:
\(T_{\ast}^k(\psi; c; v) = \pi_{\leq k} (\psi(c + v) - \psi(c))\).
Furthermore
\(T_{\ast}^k(\phi; x; v) = (T_{\ast}^k(f_1; x; v), \dots, T_{\ast}^k(f_n; x; v)) \in \mathcal{P}(E,\IR^n)=\mathcal{P}(E)\tensor \IR^n\).

Now Faà di Bruno in composition form (Theorem \ref{thm:fdb-comp}) yields
\(T_{\ast}^k(f_1 \cdots f_n; x; v) =  \pi_{\leq k} (\prod_{i=1}^n (c_i + T_{\ast}^k(f_i; x; v)) - \prod_{i=1}^n c_i)\).
Adding \(\prod_{i=1}^n c_i\) on both sides gives the claimed identity:
\[
      T^k(f_1 \cdots f_n;\, x;\, v) = \pi_{\leq k}\bigl(T^k(f_1;\, x;\, v) \cdots T^k(f_n;\, x;\, v)\bigr).
\]

Ad 2) The decomposition formula is derived by applying \(\Pol_k\) to
both sides of (1). We hence find: \[
    D(f_1 \cdots f_n; x)
  =
    \Pol_k \prod_{i=1}^n \sum_{m=0}^k \frac{1}{m!} D(f_i; x; \underbrace{v,\dots,v}_{m})
  =
    \Pol_k \ssum{m_1,\dots,m_n \in \IN_0 \\\\ m_1 + \dots + m_n = k} \prod_{i=1}^n  \frac{1}{m_i!} D(f_i; x; \underbrace{v,\dots,v}_{m_i}).
\]

The claim follows by the Polarization Lemma \ref{lem:pol-comp} applied
to \(B[a_1,\dots,a_n] = \prod a_i\).

Ad 3) Expanding the Taylor polynomial in multi-index form
\(T^k(f_i; x; v) = \sum_{0 \leq |\mu| \leq k} \frac{v^\mu}{\mu!} \del^\mu f_i(x)\),
we find: \[
\frac{\del^\nu F(x)}{\nu!} 
= [v^\nu] \prod_{i=1}^n \sum_{0 \leq |\mu| \leq k} \frac{v^\mu}{\mu!} \del^\mu f_i(x) 
= \sum_{\nu_1 + \cdots + \nu_n = \nu} \frac{\del^{\nu_1} f_1(x)}{\nu_1!} \cdots \frac{\del^{\nu_n} f_n(x)}{\nu_n!}.
\]

\end{proof}

\bibliographystyle{alpha-local}
\bibliography{refs}

\begin{thebibliography}{Mal66}

\bibitem[CS96]{CS1996}
Gregory~M. Constantine and Thomas~H. Savits.
\newblock A multivariate {Fa{\`a} di Bruno} formula with applications.
\newblock {\em Transactions of the American Mathematical Society},
  348(2):503--520, 1996.
\newblock URL
  \url{https://www.ams.org/journals/tran/1996-348-02/S0002-9947-96-01501-2/}.

\bibitem[dB55]{FaaDiBruno1855}
F.~Fa{\`a} di~Bruno.
\newblock Sullo sviluppo delle funzioni.
\newblock {\em Annali di Scienze Matematiche e Fisiche}, 6:479--480, 1855.

\bibitem[Din99]{Dineen1999}
Se\'an Dineen.
\newblock {\em Complex Analysis on Infinite Dimensional Spaces}.
\newblock Springer Monographs in Mathematics. Springer, London, 1999.
\newblock URL \url{https://doi.org/10.1007/978-1-4471-0869-6}.

\bibitem[EM03]{EM2003}
Luis~Hern{\'a}ndez Encinas and Juan~Mu{\~n}oz Masqu{\'e}.
\newblock A short proof of the generalized {Fa{\`a} di Bruno's} formula.
\newblock {\em Applied Mathematics Letters}, 16(7):975--979, 2003.
\newblock URL \url{https://doi.org/10.1016/S0893-9659(03)90087-4}.

\bibitem[Lan93]{LangRFA}
Serge Lang.
\newblock {\em Real and Functional Analysis}, volume 142 of {\em Graduate Texts
  in Mathematics}.
\newblock Springer, New York, 3rd edition, 1993.
\newblock URL \url{https://link.springer.com/book/10.1007/978-1-4612-0897-6}.

\bibitem[Lev06]{Levy2006}
Eliahu Levy.
\newblock Why do partitions occur in {Faa di Bruno's} chain rule for higher
  derivatives?, 2006.
\newblock arXiv preprint.
\newblock URL \url{https://arxiv.org/abs/math/0602183}.

\bibitem[Mal66]{malgrange}
Bernard Malgrange.
\newblock {\em Ideals of Differentiable Functions}.
\newblock Oxford University Press, Oxford, 1966.
\newblock Tata Institute of Fundamental Research Studies in Mathematics.
\newblock URL \url{https://archive.org/details/idealsofdifferen0000malg}.

\end{thebibliography}

\end{document}